\documentclass{IOS-Book-Article}

\usepackage{makeidx}
\usepackage{graphicx}
\usepackage{latexsym}
\usepackage{xspace}
\usepackage{euscript}
\usepackage{amsmath}
\usepackage{color}
\usepackage{amssymb}
\usepackage[latin1]{inputenc}

\usepackage[latin1]{inputenc}
\usepackage{times}
\normalfont
\usepackage[T1]{fontenc}

\usepackage{amsfonts}

\usepackage[latin1]{inputenc}
\usepackage{url,color}

\newtheorem{theorem}{Theorem}

\newtheorem{lemma}{Lemma}

\newtheorem{defi}{Definition}

\newtheorem{ex}{Example}

\newenvironment{proof}{ \noindent \textit{Proof:} }
{\hfill$\Box$\endtrivlist}

\begin{document}

\begin{frontmatter}

\title{Term Models of Horn Clauses over Rational Pavelka Predicate Logic}

\runningtitle{Term Models of Horn Clauses over Rational Pavelka Predicate Logic}

\runningauthor{Costa}

\author[A,B]{\fnms{Vicent} \snm{Costa}},
\author[A,B,C]{\fnms{Pilar} \snm{Dellunde}}

\address[A]{Universitat Aut\`onoma de Barcelona\\
}

\vspace{-0.4cm}

\address[B]{Artificial Intelligence Research Institute (IIIA - CSIC)\\
Campus UAB, 08193 Bellaterra, Catalonia \\
}

\vspace{-0.4cm}

\address[C]{Barcelona Graduate School of Mathematics\\
\email{vicent@iiia.csic.es}\\
\email{pilar.dellunde@uab.cat}\\}

\begin{abstract}
This paper is a contribution to the study of the universal Horn fragment of predicate fuzzy logics, focusing on the proof of the existence of free models of theories of  Horn clauses over Rational Pavelka predicate logic. We define the notion of a term structure associated to every consistent theory T over Rational Pavelka predicate logic and we prove that the term models of T are free on the class of all models of T. Finally, it is shown that if T is a set of Horn clauses, the term structure associated to T is a model of T. 
\end{abstract}


\end{frontmatter}

\maketitle

\section{Introduction}

\noindent Free models and Horn clauses have a relevant role in classical logic and logic programming. On the one hand, free models, which appeared first in category theory (see for instance \cite[Def. 4.7.17]{BaWe98}), are crucial in universal algebra and, thereby, in model theory. In the context of logic programming, free structures, introduced in \cite{GoThWaWr75} and also named \emph{initial} (as for instance in \cite[Def. 2.1 (i)]{Mak87}), are important in logic programming, since they  allow a procedural interpretation of the programs and admitting free structures makes reasonable the \emph{negation as failure} (see for instance \cite{Mak87}). In the context of abstract data types, Tarlecki \cite{Tar85} characterizes abstract algebraic institutions which admit free constructions. On the other hand, the significant importance of Horn clauses in classical logic was detailed in \cite{Hod93}, while it is well-known that Horn clauses are used both as a specification and as a programming language in Prolog, the most common language in logic programming. \smallskip

In the context of fuzzy logics, several definitions of Horn clause have been proposed in the literature, but there is not a canonical one yet. An extensive and important work in predicate fuzzy logics has been done by B\v{e}lohl\'avek and Vychodil (see \cite{Be02,Be03,BeVy05,BeVic06,BeVic06b,Vy15}). Even if the work of these authors also adopts Pavelka-style, it differs from our approach: we do not restrict Horn clauses to fuzzy equalities and we work in the general semantics of \cite{Ha98}.
Another approach is shown in \cite{DuPra96}, where Dubois and Prade discuss different possibilities of defining \emph{fuzzy rules} and they show how these different semantics can be captured in the framework of fuzzy set theory and possibility theory. We find also that, in the context of fuzzy logic programming, there is a rich battery of proposals of Horn clauses which differ depending on the programming approach chosen. Some reference here are \cite{Voj01, Ebra01}. \smallskip

With the goal of developing a systematic study of the universal Horn fragment of predicate fuzzy logics from a model-theoretic point of view, we took in \cite{CoDe16} the syntactical definition of Horn clause of classical logic. Starting by this general and basic definition we studied the existence of  free models of theories of Horn clauses in MTL$\forall$. As a generalisation of a group-theoretic construction, Mal'tsev showed in classical logic that any theory of Horn clauses has a free model. In the present paper, a definition of Horn clause in RPL$\forall$ using evaluated formulas is introduced. Consequently, we prove the existence of free models of theories of RPL$\forall$- Horn clauses showing in RPL$\forall$ an analogous result to Mal'tsev's one. The advantage of using these RPL$\forall$-Horn clauses instead of the ones of \cite{CoDe16} lies in the fact that the former can be better settled in the context of fuzzy logic programming. For instance, from a syntactical point of view, basic RPL$\forall$-Horn clauses are a particular case of the clauses used in \cite{Cao04}\smallskip

The paper is organized as follows. Section 2 contains the preliminaries on RPL$\forall$. In Section 3 we introduce the definition of a term structure associated to a consistent theory and prove that when this structure is a model of the associated theory, the term structure is free on the class of all models of the theory. In Section 4 we define the notion of RPL$\forall$-Horn clause and it is shown that whenever the associated theory is a set of RPL$\forall$-Horn clauses, the term structure is a model of this theory.

\section{Preliminaries} 

\noindent In this section we introduce the basic notions and results of RPL$\forall$, the first-order extension of Rational Pavelka Logic. For an extensive presentation of RPL$\forall$ see \cite[Ch.3.3 and Ch.5.4]{Ha98} and \cite[Ch.VIII]{CiHaNo11}.

\begin{defi} \textbf{Rational Pavelka Predicate Logic \cite[Ch.VIII]{CiHaNo11}} \label{RPL} Rational Pavelka Predicate Logic RPL$\forall$ is the expansion of \L$\forall$ by adding a truth constant for each rational number $r$ in $[0,1]$ and by adding the axioms RPL1 and RPL2. The following is an axiomatic sytem for RPL$\forall$: \begin{description}

\item[($\mathbf{\text{\L}} 1$)] $\varphi\rightarrow(\psi\rightarrow\varphi)$

\item[($\mathbf{\text{\L}} 2$)] $(\varphi\rightarrow\psi)\rightarrow((\psi\rightarrow\xi)\rightarrow(\varphi\rightarrow\xi))$

\item[($\mathbf{\text{\L}} 3$)] $(\neg\psi\rightarrow\neg\varphi)\rightarrow(\psi\rightarrow\varphi)$

\item[($\mathbf{\text{\L}} 4$)] $((\varphi\rightarrow\psi)\rightarrow\psi)\rightarrow((\psi\rightarrow\varphi)\rightarrow\varphi)$

\item[($\text{RPL}1$)] $(\overline{r}\rightarrow \overline{s})\leftrightarrow\overline{r\rightarrow s}$

\item[($\text{RPL}2$)] $(\overline{r}\&\overline{s})\leftrightarrow\overline{ r\& s}$

\item[($\forall 1$)] $(\forall x)\varphi(x)\rightarrow\varphi(t)$, where the term $t$ is substitutable for $x$ in $\varphi$.

 \item[($\forall 2$)] $(\forall x)(\xi\rightarrow\varphi)\rightarrow(\xi\rightarrow(\forall x)\varphi(x))$, where $x$ is not free in $\xi$.

\end{description}

The rules are Modus Ponens and Generalization, that is, from $\varphi$ infer $(\forall x)\varphi$.  
\end{defi}

\noindent A \emph{theory} $\Phi$ is a set of sentences. We denote by $\Phi\vdash_{\text{RPL}\forall}\varphi$ the fact that $\varphi$ is provable in RPL$\forall$ from the set of formulas $\Phi$. From now on, when it is clear from the context, we will write $\vdash$ to refer to $\vdash_{\text{RPL}\forall}$. We say that a theory $\Phi$ is \emph{consistent} if $\Phi\not\vdash\overline{0}$. 

\begin{defi} 
An \emph{evaluated formula} $(\varphi,r)$ in a language of RPL$\forall$ is a formula of the form $\overline{r}\rightarrow\varphi$, where $r\in[0,1]$ is a rational number and $\varphi$ is a formula without truth constants apart from $\overline{0}$ and $\overline{1}$. We say that an evaluated formula $(\varphi,r)$ is \emph{atomic} whenever $\varphi$ is atomic. 
\end{defi}

Now we introduce the semantics of the predicate languages. Let $[0,1]_{\text{RPL}}$ be the standard RPL-algebra \cite[Def.2.2.5, Ch.II]{CiHaNo11}, a \emph{structure} for a predicate language $\mathcal{P}$ of the logic RPL$\forall$ has the form $\langle[0,1]_{\text{RPL}}, \textbf{M}\rangle$, where $ \textbf{M}=\langle M, (P_M)_{P\in Pred}, (F_M)_{F\in Func} \rangle$, $M$ is a non-empty domain; for each $n$-ary predicate symbol $P\in Pred$, $P_{\mathrm{\mathbf{M}}}$ is an $n$-ary fuzzy relation $M$, i.e., a function $M^n\rightarrow[0,1]_{\text{RPL}}$ (identified with an element of $[0,1]_{\text{RPL}}$ if $n=0$); for each $n$-ary function symbol $F\in Func$, $F_{\mathrm{\mathbf{M}}}$ is a function $M^n\rightarrow M$ (identified with an element of $M$ if $n=0$).  

An \textbf{M}-\emph{evaluation} of the object variables is a mapping $v$ which assigns an element from $M$ to each object variable. Let $v$ be an \textbf{M}-evaluation, $x$ a variable, and $a\in M$. Then by $v[x\mapsto a]$ we denote the \textbf{M}-evaluation such that $v[x\mapsto a](x)=a$ and $v[x\mapsto a](y)=v(y)$ for each object variable $y$ different from $x$. We define the \emph{values} of terms and the \emph{truth values} of formulas in the structure $\langle[0,1]_{\text{RPL}}, \textbf{M}\rangle$\space for an evaluation $v$ recursively as follows: given $F\in Func, P\in Pred$ and $c$ a connective of RPL:


\begin{itemize}

\item $||x||^{\small{[0,1]_{\text{RPL}}}}_{\mathrm{\mathbf{M}},v}=v(x)$
\item $||F(t_1,\ldots,t_n)||^{\small{[0,1]_{\text{RPL}}}}_{\mathrm{\mathbf{M}},v}=F_{\mathrm{\mathbf{M}}}(||t_1||^{\small{[0,1]_{\text{RPL}}}}_{\mathrm{\mathbf{M}},v},\ldots,||t_n||^{\small{[0,1]_{\text{RPL}}}}_{\mathrm{\mathbf{M}},v})$
\item $||P(t_1,\ldots,t_n)||^{\small{[0,1]_{\text{RPL}}}}_{\mathrm{\mathbf{M}},v}=P_{\mathrm{\mathbf{M}}}(||t_1||^{\small{[0,1]_{\text{RPL}}}}_{\mathrm{\mathbf{M}},v},\ldots,||t_n||^{\small{[0,1]_{\text{RPL}}}}_{\mathrm{\mathbf{M}},v})$
\item $||c(\varphi_1,\ldots,\varphi_n)||^{\small{[0,1]_{\text{RPL}}}}_{\mathrm{\mathbf{M}},v}=c_{[0,1]_{\text{RPL}}}(||\varphi_1||^{\small{[0,1]_{\text{RPL}}}}_{\mathrm{\mathbf{M}},v},\ldots,||\varphi_n||^{\small{[0,1]_{\text{RPL}}}}_{\mathrm{\mathbf{M}},v})$
\item $||(\forall x)\varphi||^{\small{[0,1]_{\text{RPL}}}}_{\mathrm{\mathbf{M}},v}=inf\{||\varphi||^{\small{[0,1]_{\text{RPL}}}}_{\mathrm{\mathbf{M}},v[x\rightarrow a]}\mid a\in M\}$
\item $||(\exists x)\varphi||^{\small{[0,1]_{\text{RPL}}}}_{\mathrm{\mathbf{M}},v}=sup\{||\varphi||^{\small{[0,1]_{\text{RPL}}}}_{\mathrm{\mathbf{M}},v[x\rightarrow a]}\mid a\in M\}.$
\end{itemize}

Observe that, since the universe of the standard RPL-algebra is the interval of real numbers $[0,1]$, which is complete, all the infima and suprema in the definition of the semantics of the quantifiers exist. 
 
\bigskip
\noindent For every formula  $\varphi$, possibly with variables, we write $||\varphi||^{[0,1]_{\text{RPL}}}_{\textbf{M}}=$ $$inf\{||\varphi||^{[0,1]_{\text{RPL}}}_{\textbf{M},v}\mid \text{ for every } \textbf{M}\text{ -evaluation } v\},$$ \noindent we say that $\langle[0,1]_{\text{RPL}}, \textbf{M}\rangle$ is a \emph{model of a sentence} $\varphi$ if $||\varphi||^{[0,1]_{\text{RPL}}}_{\textbf{M}}=1$; and that $\langle[0,1]_{\text{RPL}}, \textbf{M}\rangle$ is a \emph{model of a theory} $\Phi$ if  $||\varphi||^{[0,1]_{\text{RPL}}}_{\textbf{M}}=1$ for every $\varphi\in\Phi$. \bigskip

In particular, given a structure $\langle[0,1]_{\text{RPL}}, \textbf{M}\rangle$ and two formulas $\varphi$ and $\psi$:

$$||\varphi\&\psi||^{[0,1]_{\text{RPL}}}_{\textbf{M}}=max\{||\varphi||^{[0,1]_{\text{RPL}}}_{\textbf{M}}+||\psi||^{[0,1]_{\text{RPL}}}_{\textbf{M}}-1,0\}$$
$$||\varphi\rightarrow\psi||^{[0,1]_{\text{RPL}}}_{\textbf{M}}=min\{1-||\varphi||^{[0,1]_{\text{RPL}}}_{\textbf{M}}+||\psi||^{[0,1]_{\text{RPL}}}_{\textbf{M}},1\}.$$

\begin{defi}  
\label{def:mapping structures} Let $\langle[0,1]_{\text{RPL}},\mathrm{\mathbf{M}}\rangle$ and $\langle[0,1]_{\text{RPL}},\mathrm{\mathbf{N}}\rangle$ be structures, and $g$ be a mapping from $M$ to $N$. We say that $g$ is a \emph{homomorphism from} $\langle[0,1]_{\text{RPL}},\mathrm{\mathbf{M}}\rangle$ \emph{to} $\langle[0,1]_{\text{RPL}},\mathrm{\mathbf{N}}\rangle$ if for every $n$-ary function symbol $F$, any $n$-ary predicate symbol $P$ and $d_1,\ldots,d_n\in M$, \begin{itemize}
\item[(1)] $\space g(F_{\mathrm{\mathbf{M}}}(d_1,\ldots,d_n))=F_{\mathrm{\mathbf{N}}}(g(d_1),\ldots,g(d_n))$, and\newline

\item[(2)] $ P_{\mathrm{\mathbf{M}}}(d_1,\ldots,d_n)=1\Rightarrow P_{\mathrm{\mathbf{N}}}(g(d_1),\ldots,g(d_n))=1.$
 \end{itemize}

\end{defi}

Throughout the paper we assume that all our languages have a binary predicate symbol $\approx$ and we extend the axiomatic system of RPL$\forall$ in \cite[Ch.VIII]{CiHaNo11} with the following axioms of similarity and congruence.
\smallskip

\begin{defi} $\textbf{ \cite[Definitions 5.6.1, 5.6.5]{Ha98}}$ \label{def similarity} 

\begin{itemize}

\item[S1.] $(\forall x)x\approx x$ 

\item[S2.] $(\forall x)(\forall y)(x\approx y\rightarrow y\approx x$) 

\item[S3.] $(\forall x)(\forall y)(\forall z)(x\approx y \& y\approx z\rightarrow x\approx z)$  \end{itemize} 
\begin{itemize}
\item[C1.]  For each $n$-ary function symbol  $F$, \end{itemize} 
$(\forall x_1)\dotsb(\forall x_n)(\forall y_1)\dotsb(\forall y_n)(x_1\approx y_1\&\dotsb \& x_n\approx y_n\rightarrow F(x_1,\ldots,x_n)\approx F(y_1,\ldots,y_n))$

\begin{itemize}
\item[C2.]  For each $n$-ary predicate symbol  $P$, \end{itemize} 
$(\forall x_1)\dotsb(\forall x_n)(\forall y_1)\dotsb(\forall y_n)(x_1\approx y_1\&\dotsb \& x_n\approx y_n\rightarrow (P(x_1, \ldots, x_n)\leftrightarrow P(y_1,\ldots, y_n))$ 
\end{defi}

\begin{defi} \label{provability}
Let $\Phi$ be a theory over RPL$\forall$, $\varphi$ a formula in a language of RPL$\forall$ and $r\in[0,1]$ a rational number. \newline \newline
(i) The \emph{truth degree} of $\varphi$ over $\Phi$ is $||\varphi||_{\Phi}=$
$$inf\{ ||\varphi||^{[0,1]_{\emph{RPL}}}_{\emph{\textbf{M}}}\mid \langle[0,1]_{\emph{RPL}},\emph{\textbf{M}}\rangle \text{ is a model of } \Phi\}.$$
(ii) The \emph{provability degree} of $\varphi$ over $\Phi$ is$$|\varphi|_{\Phi}=sup\{r\mid \Phi\vdash\overline{r}\rightarrow\varphi\}.$$
\end{defi}

\begin{theorem} \textbf{Pavelka-style completeness \cite[Th.5.4.10]{Ha98}} \label{pavelka}
Let $\Phi$ be a theory over RPL$\forall$ and $\varphi$ a formula in a language of RPL$\forall$. Then, $|\varphi|_{\Phi}=||\varphi||_{\Phi}$.
\end{theorem}

\section{Term structures} 

\noindent In this section we introduce the notion of term structure associated to a consistent theory $\Phi$ over RPL$\forall$, and prove that whenever the term structure is a model of $\Phi$, the structure is free on the class of models of $\Phi$. Term structures have been used extensively in classical logic, for instance, to prove the satisfiability of a set of consistent sentences (see for example \cite[Ch.V]{EbiFlu94}).

\begin{defi}\label{relacio}   
Let $\Phi$ be a consistent theory, we define a binary relation on the set of terms, denoted by $\sim$, in the following way: For every terms $t_1,t_2$,

\begin{center}
$t_1\sim t_2$ if and only if $|t_1\approx t_2|_{\Phi}=1$. 
\end{center}
\end{defi}

Using Axioms $\forall1$, S1, S2 and S3, it can be proven that $\sim$ is an equivalence relation. Next lemma, which states that the equivalence relation $\sim$ is compatible with the symbols of the language, is proved using Axioms $\forall1$, C1, C2 and  \cite[Remark 3.18]{Ha98}.

\begin{lemma} \label{compatible}
For any consistent theory $\Phi$, the following holds: If $t_i\sim t'_i$ for every $1\leq i\leq n$, then \begin{itemize}

\item[(i)] For any $n$-ary function symbol $F$, $F(t_1,\ldots,t_n)\sim F(t'_1,\ldots,t'_n)$. 

\item[(ii)] For any $n$-ary predicate symbol $P$ and rational number $r\in[0,1]$, \begin{center}

$|(\overline{r}\rightarrow P(t_1,...,t_n))\leftrightarrow (\overline{r}\rightarrow P(t'_1,...,t'_n))|_{\Phi}=1$ 

\end{center}

\end{itemize}

\end{lemma}

From now on, for any term $t$ we denote by $\overline{t}$ the $\sim$-class of $t$. 

\begin{defi}\textbf{Term Structure} \label{term structure}
Let $\Phi$ be a consistent theory. We define the following structure $\langle[0,1]_{\text{RPL}},\mathrm{\mathbf{T}}^{\Phi}\rangle$, where $T^{\Phi}$ is the set of all equivalence classes of the relation $\sim$ and 
\begin{itemize}

\item For any $n$-ary function symbol $F$ and terms $t_1,\ldots,t_n$, 
$$F_{\mathrm{\mathbf{T}}^{\Phi}}(\overline{t_1},\ldots,\overline{t_n})=\overline{F(t_1,\ldots,t_n)}$$

\item For any $n$-ary predicate symbol $P$ and terms $t_1,\ldots,t_n$, 
 $$ P_{\mathrm{\mathbf{T}}^{\Phi}}(\overline{t_1},\ldots,\overline{t_n})=|P(t_1,\ldots,t_n)|_{\Phi} $$

\end{itemize}
We call $\langle[0,1]_{\text{RPL}},\mathrm{\mathbf{T}}^{\Phi}\rangle$ the \emph{term structure associated to $\Phi$}.

\end{defi}

Notice that for $0$-ary functions, that is, for individual constants, $c_{\mathrm{\mathbf{T}}^{\Phi}}=\overline{c}$. Given a consistent theory $\Phi$, let $e^{\Phi}$ be the following $\mathrm{\mathbf{T}}^{\Phi}$-evaluation: $e^{\Phi}(x)=\overline{x}$ for every variable $x$. We call $e^{\Phi}$ the \emph{canonical evaluation of} $\langle[0,1]_{\text{RPL}},\mathrm{\mathbf{T}}^{\Phi}\rangle$.

\begin{lemma} \label{terms} 
Let $\Phi$ be a consistent theory, the following holds: \begin{itemize}

\item[(i)] For any term $t$, $||t||^{[0,1]_{\emph{RPL}}}_{\mathrm{\mathbf{T}}^{\Phi},e^{\Phi}}=\overline{t}$.

\item[(ii)] For any atomic formula $\varphi$, $||\varphi||^{[0,1]_{\emph{RPL}}}_{\mathrm{\mathbf{T}}^{\Phi},e^{\Phi}}=1$ if and only if $|\varphi|_{\Phi}=1$. 

\item[(iii)] For any evaluated atomic formula $(\varphi,s)$, $||(\varphi,s)||^{[0,1]_{\emph{RPL}}}_{\mathrm{\mathbf{T}}^{\Phi},e^{\Phi}}=1$ if and only if $|(\varphi,s)|_{\Phi}=1$.

\end{itemize}  
\end{lemma}

\begin{proof}
The proofs of (i) and (ii) are straightforward. Regarding (iii), let $(\varphi,s)=(P(t_1\ldots,t_n),s)$, we have: 

\medskip$\begin{array}{lr} ||(P(t_1\ldots,t_n),s)||^{[0,1]_{\text{RPL}}}_{\mathrm{\mathbf{T}}^{\Phi},e^{\Phi}}=1 & \text{iff} 
\\[1ex]  s\leq||P(t_1\ldots,t_n)||^{[0,1]_{\text{RPL}}}_{\mathrm{\mathbf{T}}^{\Phi},e^{\Phi}} & \text{iff}
\\[1ex]  s\leq P_{\textbf{T}^{\Phi}}(\overline{t_1}\ldots,\overline{t_n})& \text{iff}
\\[1ex]  s\leq |P(t_1\ldots,t_n)|_{\Phi}  \text{ iff } |\overline{s}\rightarrow P(t_1,\ldots, t_n)|_{\Phi}=1. & \end{array}$

\smallskip

\noindent The last equivalence is proved from \cite[Remark 3.18]{Ha98}.  \end{proof}

\bigskip

Since the simplest well-formed formulas are atomic formulas, Lemma \ref{terms} (ii) can be read as saying that term structures are minimal with respect to atomic formulas. By Theorem \ref{pavelka}, $|\varphi|_{\Phi}=||\varphi||_{\Phi}$ and, by Lemma \ref{terms} (ii), the term structure $\langle[0,1]_{\text{RPL}},\mathrm{\mathbf{T}}^{\Phi}\rangle$ only assigns the truth value $1$ to those atomic formulas that have $1$ as their truth value in every model $\langle[0,1]_{\text{RPL}},\textbf{M}\rangle$ of $\Phi$. By a similar argument, Lemma \ref{terms} (iii) states that the term structure $\langle[0,1]_{\text{RPL}},\mathrm{\mathbf{T}}^{\Phi}\rangle$ is minimal with respect to evaluated atomic formulas.

From an algebraic point of view, the minimality of the term structure is revealed by the fact that the structure is \emph{free}. The following theorem proves that in case that the term structure associated to a theory is a model of that theory, the term structure is free.

Working in predicate fuzzy logics (and, in particular, in RPL$\forall$) allows to define the term structure associated to a theory using similarities instead of crisp identities. This leads us to a notion of free structure restricted to the class of reduced models of that theory. Remember that \emph{reduced structures} are those whose Leibniz congruence is the identity. By \cite[Lemma 20]{De12}, a structure $\langle[0,1]_{\text{RPL}},\mathrm{\mathbf{M}}\rangle$ is reduced iff it has the \emph{equality property} (EQP) (that is, for any $d,e\in M$, $ || d\approx e||^{[0,1]_{\text{RPL}}}_{\mathrm{\mathbf{M}}}=1$ iff $d=e$). Observe that, by using Definitions \ref{relacio} and \ref{term structure} and the fact that $\sim$ is an equivalence relation, it can be proven that $\langle[0,1]_{\text{RPL}},\mathrm{\mathbf{T}}^{\Phi}\rangle$ is a reduced structure.

\begin{theorem} \label{free thm}
Let $\Phi$ be a consistent theory such that $\langle[0,1]_{\emph{RPL}},\mathrm{\mathbf{T}}^{\Phi}\rangle$ is a model of $\Phi$. Then $\langle[0,1]_{\emph{RPL}},\mathrm{\mathbf{T}}^{\Phi}\rangle$ is free on the class of all the reduced models $\langle[0,1]_{\emph{RPL}},\mathrm{\mathbf{N}}\rangle$ of $\Phi$. That is, for every reduced model of $\Phi$ $\langle[0,1]_{\emph{RPL}},\mathrm{\mathbf{N}}\rangle$ and every $\mathrm{\mathbf{N}}$-evaluation $v$, there is a unique homomorphism $g$ from $\langle[0,1]_{\emph{RPL}},\mathrm{\mathbf{T}}^{\Phi}\rangle$ to $\langle[0,1]_{\emph{RPL}},\mathrm{\mathbf{N}}\rangle$ such that for every variable $x$, $g(\overline{x})=v(x)$. 
 \end{theorem}

\begin{proof} Let $\langle[0,1]_{\text{RPL}},\mathrm{\mathbf{N}}\rangle$ be a reduced model of $\Phi$ and $v$ an $\textbf{N}$-evaluation. 
We define $g$ by: $g(\overline{t})=|| t ||^{[0,1]_{\text{RPL}}}_{\mathrm{\mathbf{N}},v}$ for every term $t$. We show that $g$ is the claimed homomorphism. \newline 

Let us first check that $g$ is well-defined. Let $t_1,t_2$ be terms with $\overline{t_1}=\overline{t_2}$, i.e., $t_1\sim t_2$, that is, $|t_1\approx t_2|_{\Phi}=1$. From Theorem \ref{pavelka} we have $||t_1\approx t_2||_{\Phi}=1$. Since $||\Phi||^{[0,1]_{\text{RPL}}}_{\mathrm{\mathbf{N}}}=1$, it follows that $||t_1\approx t_2||^{[0,1]_{\text{RPL}}}_{\mathrm{\mathbf{N}}}=1$ and, in particular, $||t_1\approx t_2||^{[0,1]_{\text{RPL}}}_{\mathrm{\mathbf{N}},v}=1$. From this and the fact that $\langle[0,1]_{\text{RPL}},\mathrm{\mathbf{N}}\rangle$ is reduced, we deduce, by \cite[Lemma 20]{De12}, that $||t_1||^{[0,1]_{\text{RPL}}}_{\mathrm{\mathbf{N}},v}=||t_2||^{[0,1]_{\text{RPL}}}_{\mathrm{\mathbf{N}},v}$, i.e., $g(\overline{t_1})=g(\overline{t_2})$. \smallskip

The task is now to see that $g$ satisfies the conditions (1) and (2) of Definiton \ref{def:mapping structures}. For any $0$-function symbol $c$, $ c_{\mathrm{\mathbf{T}}^{\Phi}}=\overline{c}=c_{\textbf{N}}$ by Definition \ref{term structure}. Let $\overline{t_1},\ldots,\overline{t_n}\in T^{\Phi}$ and $F$ be an $n$-ary function symbol, $F_{\mathrm{\mathbf{T}}^{\Phi}}(\overline{t_1},\ldots,\overline{t_n})=\overline{F(t_1,\ldots,t_n)}$ by Definition \ref{term structure}. Then, by the definition of $g$,

\begin{center} $g(F_{\mathrm{\mathbf{T}}^{\Phi}}(\overline{t_1},\ldots,\overline{t_n}))=g(\overline{F(t_1,\ldots,t_n)})=$\end{center}
\begin{center} $F_{\textbf{N}}(|| t_1||^{[0,1]_{\text{RPL}}}_{\mathrm{\mathbf{N}},v},\ldots ,|| t_n ||^{[0,1]_{\text{RPL}}}_{\mathrm{\mathbf{N}},v})=F_{\textbf{N}}(g(\overline{t_1}),\ldots,g(\overline{t_n}))$. \end{center}

\smallskip

Let $P$ be an $n$-ary predicate symbol such that $P_{\mathrm{\mathbf{T}}^{\Phi}}(\overline{t_1},\ldots,\overline{t_n})=1$. By Definition \ref{term structure} and Theorem \ref{pavelka}, $1=P_{\mathrm{\mathbf{T}}^{\Phi}}(\overline{t_1},\ldots,\overline{t_n})=|P(t_1,\ldots,t_n)|_{\Phi}=||P(t_1,\ldots,t_n)||_{\Phi}$.

\noindent Consequently, $||P(t_1,\ldots,t_n)||^{[0,1]_{\text{RPL}}}_{\mathrm{\mathbf{N}}}=1$, because $||\Phi||^{[0,1]_{\text{RPL}}}_{\mathrm{\mathbf{N}}}=1$. Thus \newline
$|| P(t_1,\ldots,t_n)||^{[0,1]_{\text{RPL}}}_{\mathrm{\mathbf{N}},v}=1$. Therefore $P_{\mathrm{\mathbf{N}}}(|| t_1||^{[0,1]_{\text{RPL}}}_{\mathrm{\mathbf{N}},v},\ldots ,|| t_n ||^{[0,1]_{\text{RPL}}}_{\mathrm{\mathbf{N}},v})=1$, that is, $P_{\mathrm{\mathbf{N}}}(g(\overline{t_1}),\ldots,g(\overline{t_n}))=1$.

\smallskip

Finally, since the set $\{\overline{x}\mid x\text{ is a variable}\}$ generates the universe $T^{\Phi}$ of the term structure associated to $\Phi$, $g$ is the unique homomorphism such that for every variable $x$, $g(\overline{x})=v(x)$.   
\end{proof}

\bigskip

Observe that in languages in which the similarity symbol is interpreted by the crisp identity, by using an analogous argument to the one in Theorem \ref{free thm}, we obtain that the term structure is free in the class of all the models $\langle[0,1]_{\text{RPL}},\mathrm{\mathbf{M}}\rangle$ of the theory and not only in the class of the reduced ones.

\section{RPL$\forall$-Horn Clauses}

\noindent In the previous section we have seen that if the term structure associated to a theory $\Phi$ is a model of $\Phi$, then the structure is free in the class of all models of $\Phi$. In this section, we show in Theorem \ref{model} that whenever $\Phi$ is a theory of RPL$\forall$-Horn clauses, $\langle[0,1]_{\text{RPL}},\textbf{T}^{\Phi}\rangle$ is a model of $\Phi$. Theorem \ref{model} gains in interest if we realize that it proves (using Theorem \ref{free thm}) the existence of free models of theories of RPL$\forall$-Horn clauses. Let us first define the notion of RPL$\forall$-Horn clauses. \smallskip

In predicate classical logic, a \emph{basic Horn formula} is a formula of the form $ \alpha_{1}\wedge\dotsb \wedge\alpha_{n}\rightarrow\beta$, where $n$ is a natural number and $\alpha_1,\ldots,\alpha_n,\beta$ are atomic formulas. Notice that there is not a unique way to extend this definition in fuzzy logics, where we have different conjunctions and implications. In this section we present one way to define Horn clauses over RPL$\forall$ extending the classical definition.

\smallskip

\begin{defi}\textbf{Basic RPL$\forall$-Horn Formula}\label{strong basic} A \emph{basic \emph{RPL}$\forall$-Horn formula} is a formula of the form
 $$(\alpha_1,r_1)\&\dotsb\&(\alpha_n,r_n)\rightarrow(\beta,s)$$ 

where $(\alpha_1,r_1)\ldots,(\alpha_n,r_n),(\beta,s)$ are evaluated atomic formulas and $n$ is a natural number. Observe that $n$ can be $0$. In that case the basic RPL$\forall$-Horn formula is an evaluated atomic formula.

\end{defi} 

\begin{defi} \textbf{Quantifier-free RPL$\forall$-Horn Formula} \label{qf Horn} A \emph{quantifier-free \emph{RPL}$\forall$-Horn formula} is a formula of the form $\phi_1\&\dotsb\&\phi_m$, where $m$ is a natural number and $\phi_i$ is a basic RPL$\forall$-Horn formula for every $1\leq i\leq m$. 
\end{defi}

\begin{defi} \textbf{RPL$\forall$-Horn Clause}     \label{Horn} 
A \emph{\emph{RPL}$\forall$-Horn clause} is a formula of the form $Q\gamma$, where $Q$ is a (possibly empty) string of universal quantifiers $(\forall x)$ and $\gamma$ is a quantifier-free RPL$\forall$-Horn formula. \end{defi}


\begin{ex} \label{example}
Let $\mathcal{P}$ be a predicate language with a unary predicate symbol $P$, a binary predicate symbol $R$ and $a$ an individual constant. The following formulas are examples of RPL$\forall$-Horn clauses: \begin{itemize}

\item[(1)] $(P(a),0.5)$,

\item[(2)] $(P(a),0.6)\&(R(a,x),0.3)$,

\item[(3)] $(P(a),0.5)\rightarrow(R(a,a),0.1)$,

\item[(4)] $(P(a),0.6)\&(R(a,x),0.3)\rightarrow(P(x),0.8)$,

\item[(5)] $(\forall x)((P(x),0.6)\&(R(a,x),0.3))$,

\item[(6)] $(\forall x)((P(x),0.6)\&(R(a,x),0.3)\rightarrow(P(a),0.9))$.

\end{itemize}
\end{ex}

Observe that, in general, RPL$\forall$-Horn clauses are not evaluated, only the atomic RPL$\forall$-Horn clauses  are evaluated formulas.

A weak version of RPL$\forall$-Horn clauses can be introduced by substituting each strong conjunction $\&$ appearing in the formula by the weak conjunction $\wedge$. Although in this paper we do not present this weak version, all the results we prove are also true for weak RPL$\forall$-Horn clauses. In classical logic, the set of all Horn clauses is recursively defined, because the formula $(\forall x)(\varphi\wedge\psi)$ is logically equivalent to $(\forall x)\varphi\wedge(\forall x)\psi$. In RPL$\forall$ these two formulas are also logically equivalent, so the set of the weak version of fuzzy RPL$\forall$-Horn clauses is recursively definable. However, this is not the case for fuzzy RPL$\forall$-Horn clauses. Indeed, let $P$ and $R$ be unary predicate symbols, consider the structure $\langle[0,1]_{\text{RPL}},\textbf{M}\rangle$ such that  $M=\{a,b\}$, $ P_{\textbf{M}}(a)=R_{\textbf{M}}(b)=0.4$ and $P_{\textbf{M}}(b)=R_{\textbf{M}}(a)=0.7$. Then, $||(\forall x)((P(x),1)\&(R(x),1))||^{[0,1]_{\text{RPL}}}_{\textbf{M}}=0.1$, but $||(\forall x)((P(x),1))\&(\forall x)((R(x),1))||^{[0,1]_{\text{RPL}}}_{\textbf{M}}=0$.

We now see that for any consistent theory of RPL$\forall$-Horn clauses $\Phi$, the term structure associated to $\Phi$ is a model of $\Phi$. To show that, we need the following lemmas and the notion of rank of a formula. Our definition of rank  is a variant of the notion of \emph{syntactic degree of a formula} of \cite[Def. 5.6.7]{Ha98}). Let $\varphi$ be a formula, the \emph{rank} of $\varphi$, denoted by $rk(\varphi)$ is defined by:
\begin{itemize} 
\item $rk(\varphi)=0$ if $\varphi$ is atomic;
\item $rk(\neg\varphi)=rk((\exists x)\varphi)=rk((\forall x)\varphi)=rk(\varphi)+1$;
\item $rk(\varphi\circ\psi)=rk(\varphi)+rk(\psi)$ for every binary propositional connective $\circ$.
\end{itemize}

\noindent Note that since the set of RPL$\forall$-Horn clauses is not recursively definable, induction on the complexity of the clause cannot be applied. Hence it is applied on the rank of the clauses. Such induction can be used to prove next lemma. 

\begin{lemma} \label{Horn substitucio}
Let $\varphi$ be an RPL$\forall$-Horn clause where $x_1,\ldots,x_m$ are pairwise distinct free variables. Then, for every terms $t_1,\ldots,t_m$, the substitution $$\varphi (t_1,\ldots,t_m/x_1,\ldots,x_m)$$ is an RPL$\forall$-Horn clause.
\end{lemma}

\begin{lemma} \label{lemma th}
For any consistent theory $\Phi$ and any evaluated atomic formula $(\varphi,s)$, 
$$||(\varphi,s)||^{[0,1]_{\emph{RPL}}}_{\emph{\textbf{T}}^{\Phi}}=||(\varphi,s)||_{\Phi}.$$

\end{lemma}

\begin{proof}
It is enough to show that for any rational number $t\in[0,1]$, $||(\varphi,s)||^{[0,1]_{\text{RPL}}}_{\textbf{T}^{\Phi}}\geq t \text{ iff } ||(\varphi,s)||_{\Phi}\geq t.$ Let $t\in[0,1]$ be a rational number, we have: $$||(\varphi,s)||^{[0,1]_{\text{RPL}}}_{\textbf{T}^{\Phi}}\geq t \text{ iff } ||\overline{t}\rightarrow(\overline{s}\rightarrow\varphi)||^{[0,1]_{\text{RPL}}}_{\textbf{T}^{\Phi}}=1 \text{ iff }$$ 
$$||\overline{t}\&\overline{s}\rightarrow\varphi||^{[0,1]_{\text{RPL}}}_{\textbf{T}^{\Phi}}=1 \text{ iff } ||\varphi||^{[0,1]_{\text{RPL}}}_{\textbf{T}^{\Phi}}\geq t*_{\text{\L}}s\text{ iff}$$
$$||\varphi||^{[0,1]_{\text{RPL}}}_{\textbf{M}}\geq t*_{\text{\L}}s\text{ for every model } \langle[0,1]_{\text{RPL}},\textbf{M}\rangle \text{ of } \Phi\text{ iff}$$
$$\text{ for any model } \langle[0,1]_{\text{RPL}},\textbf{M}\rangle \text{ of } \Phi,$$
$$||\overline{t}\rightarrow(\overline{s}\rightarrow\varphi)||^{[0,1]_{\text{RPL}}}_{\textbf{M}}=1.$$
\noindent The second and latter equivalence are proved by using \cite[Def.2.2.4 (Axioms 5a and 5b)]{Ha98}. The latter expression is equivalent to $||(\varphi,s)||^{[0,1]_{\text{RPL}}}_{\textbf{M}}\geq t$ for every model $\langle[0,1]_{\text{RPL}},\textbf{M}\rangle$ of $\Phi$, i.e., $||(\varphi,s)||_{\Phi}\geq t.$ \end{proof}


\bigskip

\begin{lemma} \label{desigualtat}
For any consistent theory $\Phi$ and any evaluated atomic sentences \newline {\footnotesize $(\varphi_1,s_1),\dots,(\varphi_n,s_n)$},   {\footnotesize $$||(\varphi_1,s_1)\&\dotsb\&(\varphi_n,s_n)||^{[0,1]_{\emph{RPL}}}_{\emph{\textbf{T}}^{\Phi}}\leq||(\varphi_1,s_1)\&\dotsb\&(\varphi_n,s_n)||_{\Phi}.$$}
\end{lemma}

\begin{proof}
By Lemma \ref{lemma th}, it is clear for $n=1$. For the sake of clarity, we present the proof for the case $n=2$, but the argument is analogous for the cases with $n> 2$. First, by Lemma \ref{lemma th} we have:\newline\newline
{\footnotesize $||(\varphi_1,s_1)\&(\varphi_2,s_2)||^{[0,1]_{\text{RPL}}}_{\textbf{T}^{\Phi}}$$=||(\varphi_1,s_1)||^{[0,1]_{\text{RPL}}}_{\textbf{T}^{\Phi}}*_{\text{\L}}||(\varphi_2,s_2)||^{[0,1]_{\text{RPL}}}_{\textbf{T}^{\Phi}}=$
$||(\varphi_1,s_1)||_{\Phi}*_{\text{\L}}||(\varphi_2,s_2)||_{\Phi}.$
}

\bigskip

\noindent Since for any model {\footnotesize$\langle[0,1]_{\text{RPL}},\textbf{M}\rangle$} of $\Phi$, {\footnotesize $||(\varphi_1,s_1)||_{\Phi}\leq||(\varphi_1,s_1)||^{[0,1]_{\text{RPL}}}_{\textbf{M}}$} and {\footnotesize$||(\varphi_2,s_2)||_{\Phi}\leq||(\varphi_2,s_2)||^{[0,1]_{\text{RPL}}}_{\textbf{M}}$}, we have that for any model {\footnotesize $\langle[0,1]_{\text{RPL}},\textbf{M}\rangle$} of $\Phi$,
{\footnotesize
 $$||(\varphi_1,s_1)||_{\Phi}*_{\text{\L}}||(\varphi_2,s_2)||_{\Phi}\leq ||(\varphi_1,s_1)||^{[0,1]_{\text{RPL}}}_{\textbf{M}}*_{\text{\L}}||(\varphi_2,s_2)||^{[0,1]_{\text{RPL}}}_{\textbf{M}}=$$
 $ ||(\varphi_1,s_1)\&(\varphi_2,s_2)||^{[0,1]_{\text{RPL}}}_{\textbf{M}}.$
}
\newline 

\noindent Therefore, since $||(\varphi_1,s_1)\&(\varphi_2,s_2)||_{\Phi}$ is the infimum, we have \newline\newline {\footnotesize $||(\varphi_1,s_1)||_{\Phi}*_{\L}||(\varphi_2,s_2)||_{\Phi}\leq ||(\varphi_1,s_1)\&(\varphi_2,s_2)||_{\Phi}$. 
}

\smallskip 
\smallskip

\noindent Consequently, \newline\newline {\footnotesize $||(\varphi_1,s_1)\&(\varphi_2,s_2)||^{[0,1]_{\text{RPL}}}_{\textbf{T}^{\Phi}}\leq  ||(\varphi_1,s_1)\&(\varphi_2,s_2)||_{\Phi}$.}
\end{proof}

\smallskip

\begin{theorem} \label{model}
Let $\Phi$ be a consistent theory. For every RPL$\forall$-Horn clause $\varphi$ without free variables,  $$\text{If } |\varphi|_{\Phi}=1\text{, then } ||\varphi||^{[0,1]_{\emph{RPL}}}_{\emph{\textbf{T}}^{\Phi}}=1.$$ \end{theorem}

\begin{proof}
Let $\varphi$ be an RPL$\forall$-Horn clause without free variables. We proceed by induction on $rk(\varphi)$. \smallskip

\smallskip

\underline{$rk(\varphi)=0$.} We can distinguish three subcases: 

\smallskip

1) If $\varphi=(\psi,s)$ is an evaluated atomic formula, the statement holds by Lemma \ref{lemma th} (iii). \smallskip


2) Let $\varphi=(\psi_1,s_1)\&\dotsb\&(\psi_n,s_n)\rightarrow(\psi,s)$ be a basic RPL$\forall$-Horn formula, where $(\psi_1,s_1),\ldots,(\psi_n,s_n),(\psi,s)$ are evaluated atomic formulas. By hypothesis and Theorem \ref{pavelka} we have: 

\begin{center}$1=|(\psi_1,s_1)\&\dotsb\&(\psi_n,s_n)\rightarrow(\psi,s)|_{\Phi}=||(\psi_1,s_1)\&\dotsb\&(\psi_n,s_n)\rightarrow(\psi,s)||_{\Phi}$.\end{center}

\noindent Therefore, $||(\psi_1,s_1)\&\dotsb\&(\psi_n,s_n)||_{\Phi}\leq ||(\psi,s)||_{\Phi}$.

\smallskip

\noindent  By Lemmas \ref{lemma th} and \ref{desigualtat}, $|| (\psi,s)||^{[0,1]_{\text{RPL}}}_{\textbf{T}^{\Phi}}=||(\psi,s)||_{\Phi}$ and  
$$||(\psi_1,s_1)\&\dotsb\&(\psi_n,s_n)||^{[0,1]_{\text{RPL}}}_{\textbf{T}^{\Phi}}\leq ||(\psi_1,s_1)\&\dotsb\&(\psi_n,s_n)||_{\Phi}.$$

\noindent Therefore $||(\psi_1,s_1)\&\dotsb\&(\psi_n,s_n)||^{[0,1]_{\text{RPL}}}_{\textbf{T}^{\Phi}}\leq || (\psi,s)||^{[0,1]_{\text{RPL}}}_{\textbf{T}^{\Phi}}.$ That is, $$ ||(\psi_1,s_1)\&\dotsb\&(\psi_n,s_n)\rightarrow(\psi,s)||^{[0,1]_{\text{RPL}}}_{\textbf{T}^{\Phi}}=1.$$

\smallskip

3) If $\varphi=\phi_1\&\dotsb\&\phi_m$ is a conjunction of basic RPL$\forall$-Horn formulas, \begin{center}

$||\phi_1\&\dotsb\&\phi_m||^{[0,1]_{\text{RPL}}}_{\mathrm{\mathbf{T}}^{\Phi}}=1$ iff \end{center} 
\begin{center}
$||\phi_i||^{[0,1]_{\text{RPL}}}_{\mathrm{\mathbf{T}}^{\Phi}}=1$ for every $1 \leq i \leq m$. 
\end{center}

From 1) and 2), $|\phi_i|_{\Phi}=1$ for every $1\leq i\leq m$ and thus $|\phi_1\&\dotsb\&\phi_m|_{\Phi}=1$. \newline


\underline{$rk(\varphi)=n+1$.} Let $\varphi=(\forall x)\psi$ be such that $\psi$ is an RPL$\forall$-Horn clause of rank $n$. Assume inductively that for any RPL$\forall$-Horn clause without free variables $\xi$ of rank less or equal than $n$ and such that $|\xi|_{\Phi}=1$, $||\xi||^{[0,1]_{\text{RPL}}}_{\mathrm{\mathbf{T}}^{\Phi}}=1$. By assumption and Axiom $\forall 1$, \begin{center}

$\Phi\vdash(\forall x)\psi\rightarrow\psi(t/x)$ for every term $t$.

\end{center}

From Axiom \L 2, $sup\{r\mid\Phi\vdash\overline{r}\rightarrow\varphi\}=1$ implies that $sup\{r\mid\Phi\vdash\overline{r}\rightarrow\psi(t/x)\}=1$ for any term $t$. That is, $|\psi(t/x)|_{\Phi}=1$ for every term $t$. \smallskip

Since $\psi$ has rank $n$ and is an RPL$\forall$-Horn clause by Lemma \ref{Horn substitucio}, we can apply the inductive hypothesis and conclude that $||\psi(t/x)||^{[0,1]_{\text{RPL}}}_{\mathrm{\mathbf{T}}^{\Phi}}=1$ for any term $t$. So, by Lemma \ref{terms} (i), $||\psi(x)||^{[0,1]_{\text{RPL}}}_{\mathrm{\mathbf{T}}^{\Phi},v[x\mapsto\overline{t}]}=1$ for every element $\overline{t}$ of the domain, and thus we get $||(\forall x)\psi||^{[0,1]_{\text{RPL}}}_{\mathrm{\mathbf{T}}^{\Phi}}=1$.
\end{proof}

\section{Conclusions and Future Work}

\noindent The present paper is another step towards a systematic study of theories of Horn clauses over predicate fuzzy logics from a model-theoretic point of view, a study that we started in \cite{CoDe16} and which is still in progress. In particular, here we  have proved the existence of free models of theories of Horn clauses in RPL$\forall$. \smallskip

Future work will be devoted to study the broad approach taken in \cite[Ch.8]{CiHaNo11} to fuzzy logics with enriched languages. We shall see if RPL$\forall$-Horn clauses introduced here can be generalized to that logics with enriched languages. Later, since one of our next goals is to solve the open problem (formulated by Cintula and H\'ajek in \cite{CiHa10}) about the characterization of theories of fuzzy Horn clauses in terms of quasivarieties, we will analyze quasivarieties and try to define them in the context of fuzzy logics using recent results on products over fuzzy logics like \cite{De12}.  \smallskip

Herbrand structures have been important in model theory and in the foundations of logic programming. Therefore, as a continuation of the present work, we would like to characterize the free Herbrand model in the class of the Herbrand models of theories of RPL$\forall$-Horn clauses without equality. Finally, we will focus on a generalization of Herbrand structure, \emph{fully named models}, in order to show that two types of minimality for these models (specifically free models and $A$-generic models) are equivalent.

  \section*{Acknowledgments}
\noindent We would like to thank the referees for their useful comments. This project has received funding from the European Union's Horizon 2020 research and innovation program under the Marie Sklodowska-Curie grant agreement No 689176 (SYSMICS project). Pilar Dellunde is also partially supported by the project RASO TIN2015-71799-C2-1-P (MINECO/FEDER) and the grant 2014SGR-118 from the Generalitat de Catalunya.

\end{document}